\documentclass{amsart}
\usepackage[T1]{fontenc}
\usepackage{hyperref}
\usepackage{texdraw}

\makeatletter

\newfont{\gothic}{eufm10 scaled 1100}

\theoremstyle{plain}    
\newtheorem{thm}{Theorem}[section]
\numberwithin{equation}{section} %% Comment out for sequentially-numbered
\numberwithin{figure}{section} %% Comment out for sequentially-numbered
\theoremstyle{plain}    
\theoremstyle{plain}    
\theoremstyle{plain}    
\theoremstyle{plain}
\newtheorem{lem}[thm]{Lemma} %%Delete [thm] to re-start numbering
\theoremstyle{plain}    
\newtheorem{prop}[thm]{Proposition} %%Delete [thm] to re-start numbering
\theoremstyle{plain}    
\newtheorem{Def}[thm]{Definition} %%Delete [thm] to re-start numbering
\theoremstyle{remark}
\newtheorem{rem}[thm]{Remark}
\theoremstyle{remark}

\usepackage[arrow,matrix,curve,ps]{xy}
\xyoption{dvips}
\CompileMatrices

\makeatother

\begin{document}

\title{An asymptotic version of Dumnicki's algorithm for linear systems in $\mathbb{CP}^2$}

\date{\today}

\author{Thomas Eckl}



%\thanks{}

\address{Thomas Eckl, Department of Mathematical Sciences, The University of Liverpool, Mathematical
               Sciences Building, Liverpool, L69 7ZL, England, U.K.}

\email{thomas.eckl@liv.ac.uk}

\urladdr{http://pcwww.liv.ac.uk/~eckl/}

\maketitle

\begin{abstract}
Using Dumnicki's approach to showing non-specialty of linear systems consisting of plane curves with 
prescribed multiplicities in sufficiently general points on $\mathbb{P}^2$ we develop an asymptotic method 
to determine lower bounds for Seshadri constants of general points on $\mathbb{P}^2$. With this method we 
prove the lower bound $\frac{4}{13}$ for $10$ general points on $\mathbb{P}^2$.
\end{abstract}

%\tableofcontents

\pagestyle{myheadings}
\markboth{THOMAS ECKL}{AN ASYMPTOTIC VERSION OF DUMNICKI'S ALGORITHM}

\setcounter{section}{-1}

\section{Introduction}

\noindent A celebrated conjecture of Nagata~\cite{Nag59} predicts that every curve in 
$\mathbb{P}^2 = \mathbb{CP}^2$ going through $r > 9$ 
very general points with multiplicity at least $m$ has degree $d \geq \sqrt{r} m$. Cast in the language of 
Seshadri constants, Nagata claimed in effect that
\[ H - \sqrt{\frac{1}{r}} \sum_{j=1}^r E_j \]
is a nef divisor on $\widetilde{X} = \mathrm{Bl}_r(\mathbb{P}^2)$, the blowup of $\mathbb{P}^2$ in the $r$ points,
where $H$ is the pullback of a line in $\mathbb{P}^2$ and $E_j$ are the exceptional divisors over the blown up 
points.

\noindent It is well known that Nagata's conjecture is implied by another 
conjecture of Harbourne and Hirschowitz about spaces $\mathcal{L}_d(m^r)$ of 
plane curves of given degree $d$ and multiplicity at least $m$ at $r$ general
points \cite{Mir99, CilMir01}. This conjecture tries to detect those of the 
spaces $\mathcal{L}_d(m^r)$ which do not have the expected dimension
\[ \max (-1, \frac{d(d+3)}{2} - r \cdot \frac{m(m+1)}{2}). \]

\noindent In~\cite{Eck05b} the author showed that it is not necessary to know all cases of the 
Harbourne-Hirschowitz conjecture in order to prove Nagata's conjecture:
\begin{thm}[\cite{Eck05b},Thm.5.1] \label{Eckl-thm}
Let $r > 9$ be an integer and $(d_i, m_i)$ a sequence of pairs of positive 
integers such that 
$\frac{d_i^2}{m_i^2 \cdot r} 
 \stackrel{i \rightarrow \infty}{\longrightarrow} \frac{1}{a^2} \geq 1$ and the
space 
$\mathcal{L}_{d_i}((m_i+1)^r)$ has expected dimension $\geq 0$. Then 
\[ H - a \cdot \sqrt{\frac{1}{r}} \sum_{j=1}^r E_j \]
is nef on $\widetilde{X}$. In particular, Nagata's
conjecture is true for $r$ general points in $\mathbb{P}^2$, if $a = 1$.
\end{thm}

\noindent In this paper we want to use Dumnicki's Reduction Algorithm \cite{DJ05, D06} to prove the 
non-specialty of linear systems $\mathcal{L}_d(m^r)$, as needed in the theorem. Dumnicki's new idea was to
consider linear systems of curves not only going through certain points with at least certain multiplicities, but
also the curve equation should contain only monomials from a certain subset of all monomials of degree 
$\leq d$. He was able to give non-specialty criteria for such linear systems, including the following:
\begin{prop}[\textbf{Dumnicki's non-specialty criterion}]
Let $m \in \mathbb{N}$ and let $D \subset \mathbb{N}^2$ such that $\# D = \left( \begin{array}{c} 
                                                                                 m+1\\ 2         
                                                                                 \end{array}\right)$.
Consider the linear system $L = \mathcal{L}_D(m)$ of those curve equations 
$\sum_{(\alpha,\beta) \in D} c_{\alpha,\beta} x^\alpha y^\beta$, $c_{\alpha,\beta} \in \mathbb{C}$, which 
pass through a given point with multiplicity at least $m$. Then $L$ is non-special if and only if the points in $D$
do not lie on a curve of degree $m-1$ in $\mathbb{R}^2$.

\noindent In particular, $L$ is non-special if there are $m$ parallel lines $l_1, \ldots, l_m$ containing 
$1, \ldots, m$ points in $D$.
\end{prop}
\begin{proof}
See \cite[Prop.12]{D06}. The last statement follows from B\'ezout's Theorem.
\end{proof}

\noindent Furthermore, Dumnicki devised a recursive procedure showing the non-specialty of linear systems
$\mathcal{L}_D(m_1, \ldots, m_r)$ if it terminates in the correct way:
\begin{thm}[\textbf{Dumnicki's reduction algorithm}]
Let $m_1, \ldots, m_{p-1}, m_p \in \mathbb{N}^\ast$, let $D \subset \mathbb{N}^2$, and let 
\[ F: \mathbb{R}^2 \ni (a_1,a_2) \mapsto r_0 + r_1 a_1 + r_2 a_2 \in \mathbb{R},\ r_0, r_1, r_2 \in \mathbb{R}, \]
be an affine function. Let
\[ \begin{array}{rcl} 
   D_1 & := & \left\{ (a_1,a_2) \in D | F(a_1, a_2) < 0 \right\}, \\
   D_2 & := & \left\{ (a_1,a_2) \in D | F(a_1, a_2) > 0 \right\}. 
   \end{array} \]
If $D_1 \cup D_2 = D$ and $L_1 := \mathcal{L}_D(m_1, \ldots, m_{p-1})$ is non-special of dimension $\geq 0$, 
$L_2 := \mathcal{L}_D(m_p)$ is non-special of dimension $-1$, then $\mathcal{L}_D(m_1, \ldots, m_p)$ is
non-special of dimension $\geq 0$. 
\end{thm}
\begin{proof}
See \cite[Thm.13]{D06}. 
\end{proof}

\noindent Dumnicki used this procedure to show the Harbourne-Hirschowitz conjecture up to $m=42$, but the 
power and simplicity of it can best be seen by some easy graphical proofs of non-specialty. For example, 
Dumnicki~\cite[Ex.37]{D06} used a computer to find the following proof for non-specialty of the system
$L=\mathcal{L}_{21}(7^{\times 6}, 6^{\times 4},1)$:

$$

$$

\noindent Since the Nagata conjecture for square-free integers $r > 9$ involves irrational square roots, it
seems appropriate to look for an asymptotic version of Dumnicki's reduction algorithm. To this purpose we 
introduce the following notion:
\begin{Def}    \label{asymp-sys-Def}
Let $m_1, \ldots, m_r \in \mathbb{R}_{>0}$. A subset $P \subset \mathbb{R}^2_{\geq 0}$ contains 
asymptotically $(m_1, \ldots, m_r)$-non-special systems (of dimension $\geq d$) iff for all $\delta > 0$ and all
$k >> 0$ there exist $m_1^{(k)}, \ldots, m_r^{(k)} \in \mathbb{N}$ and a 
$D_k \subset k \cdot P \cap \mathbb{N}^2_{\geq 0}$ such that
\begin{itemize}
\item[(i)] $\mathcal{L}_{D_k}(m_1^{(k)}, \ldots, m_r^{(k)})$ is non-special (of dimension $\geq d$) and
\item[(ii)] $\left| \frac{m_i^{(k)}-km}{km_i} \right| < \delta$, $i=1, \ldots , r$.  
\end{itemize} 
\end{Def}

\noindent With this notion we prove the following method of obtaining bounds on Seshadri constants on
$\mathbb{P}^2$ (see Section~\ref{LowerBound-sec} for the proof):
\begin{prop} \label{nef-crit}
If the set $P := \{ x + y \leq 1 \} \cap \mathbb{R}^2_{\geq 0}$ contains asymptotically $(m^r)$-non-special systems
of dimension $\geq 0$ then
\[ H - m \sum_{j=1}^r E_j \]
is nef on $X$, where $X$ is the blow-up of $\mathbb{P}^2$ in $r$ very general points, $H$ is the pullback of a 
line in $\mathbb{P}^2$ to $X$, and the $E_i$ are the exceptional divisors on $X$.
\end{prop} 

\noindent To show the existence of $(m^r)$-non-special sytems we develop an asymptotic version of Dumnicki's
reduction algorithm (see Thm.~\ref{As-red-alg-Thm}), and
together with a criterion for asymptotic $(m)$-non-specialty (see Thm.~\ref{as-m-non-spec-crit}), we are able to 
give the following bound on the Seshadri constant of $10$ very general points on $\mathbb{P}^2$ (see again 
Section~\ref{LowerBound-sec} for the proof): 
\begin{thm} \label{Seshadri-bd-thm}
Let $X$ be the blow-up of $\mathbb{P}^2$ in 10 very general points, let $E_1, \ldots, E_{10}$ be the exceptional
divisors on $X$, and let $H$ be the pull back of a line in $\mathbb{P}^2$. Then the divisor
\[ H - \frac{4}{13} \sum_{i=1}^{10} E_i \]
is nef on $X$.  
\end{thm}

\noindent In recent years many authors tried to give lower bounds for the Seshadri constants of a fixed number of general 
points on algebraic surfaces, and especially on $\mathbb{P}^2$ \cite{Xu95, ST02, Harb03, HR04, HR05}. Some of these 
bounds are even better than 
$\frac{4}{13} \approx 0.307$, which is still  not really close to $\frac{1}{\sqrt{10}} \approx 0.3162$: Tutaj-Gasi\'nska 
\cite{Tut03} achieved $\frac{2}{11}\sqrt{3} \approx 0.314$, Biran \cite{Bir99} $\frac{6}{19} \approx 0.3158$, and 
Harbourne-Ro\'e \cite{HR03} even $\frac{177}{560} \approx 0.31607$. At 
least, our bound is better than what can be achieved by using the non-specialty of all non-empty linear systems
$\mathcal{L}_d(m^{10})$ up to $m \leq 42$, shown by Dumnicki~\cite{D06}: The expected dimension of
$\mathcal{L}_d(41^{10})$ is $\geq 0$ iff $d \geq 132$, hence Thm.~\ref{Eckl-thm} applied to the constant 
sequence $(132,40)$ gives the bound $\frac{40}{132} \approx 0.303$. For all other multiplicities $m \leq 42$ the 
bound gets smaller.

\noindent In any case the true interest in this bound lies in the fact that it was shown with an asymptotic method.

\vspace{0.2cm}

\noindent \textit{Acknowledgement.}
The author would like to thank Felix Sch\"uller who explained Dumnicki's techniques in his diploma thesis \cite{Sch07}.

\section{Monotone Reordering}

\noindent In this section we collect elementary, but useful facts about the \textit{monotone reordering} of 
functions:
\begin{Def}
Let $f: [a,b] \rightarrow \mathbb{R}$ be a measurable function on a closed interval $[a,b] \subset \mathbb{R}$.
Then the monotone reordering $f^\#: (0,b-a] \rightarrow \mathbb{R}$ of $f$ is defined by
\[ t \mapsto \inf \left\{ s: t \leq \mathrm{length\ of\ } \{t^\prime \in [a,b]: f(t^\prime) \leq s\} \right\}. \]
\end{Def}

\noindent This notion will be used to state the criterion of $(m)$-non-specialty (see 
Thm.~\ref{as-m-non-spec-crit}).

\begin{rem}
Let $f: [a,b] \rightarrow \mathbb{R}$ be a step function. Then $f^\#: (0, b-a] \rightarrow \mathbb{R}$ reorders its
steps such that they increase monotonely. Another example is given in the following diagram which shows the monotone 
reordering of a piecewise-linear function:

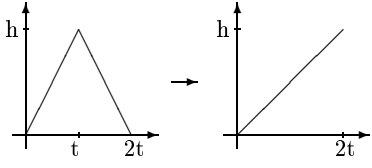
\begin{center}
\begin{picture}(150,60)(-8,-8)
\put(-5,0){\vector(1,0){55}}
\put(0,-5){\vector(0,1){55}}

\put(0,0){\line(1,2){20}}
\put(20,40){\line(1,-2){20}}

\put(-1,40){\line(1,0){2}}
\put(20,-1){\line(0,1){2}}
\put(-8,37){\small h}
\put(17,-8){\small t}
\put(37,-8){\small 2t}

\put(55,20){\vector(1,0){10}}

\put(75,0){\vector(1,0){55}}
\put(80,-5){\vector(0,1){55}}

\put(80,0){\line(1,1){40}}

\put(79,40){\line(1,0){2}}
\put(120,-1){\line(0,1){2}}
\put(72,37){\small h}
\put(117,-8){\small 2t}
\end{picture}
\end{center}
\end{rem}

\begin{prop}  \label{mon-re-inc-prop}
The monotone reordering $f^\#$ of a function $f: [a,b] \rightarrow \mathbb{R}$ is monotonely increasing and 
lower semi-continuous.
\end{prop}
\begin{proof}
If $t_1 < t_2$ then $t_2 \leq \mathrm{length\ of\ } \{ f(t^\prime) \leq s\}$ implies 
$t_1 \leq \mathrm{length\ of\ } \{ f(t^\prime) \leq s\}$, hence
\[ \begin{array}{rcccl}
f^\#(t_1) & = & \inf \left\{ s: t_1 \leq \mathrm{length\ of\ } \{ f(t^\prime) \leq s\} \right\} & & \\
             & \leq & \inf \left\{ s: t_2 \leq \mathrm{length\ of\ } \{ f(t^\prime) \leq s\} \right\} & = & f^\#(t_2). 
   \end{array} \]
Furthermore, set $s := f^\#(t)$ and assume for given $\epsilon > 0$ that
\[ \widetilde{t} \leq \mathrm{length\ of\ } \{ f(t^\prime) \leq s - \epsilon \} \]
for all $\widetilde{t} < t$. This implies $t \leq \mathrm{length\ of\ } \{ f(t^\prime) \leq s - \epsilon \}$, a 
contradiction to $f^\#(t)=s$. Hence there exists a $\overline{t} < t$ such that
$f^\#((\overline{t}, b-a)) > s - \epsilon$, and $f^\#$ is lower semi-continuous.
\end{proof}

\begin{prop}   \label{monre-ineq-prop}
Let $f_1, f_2: [a,b] \rightarrow \mathbb{R}$ be two measurable functions such that $f_1 \leq f_2$. Then
$f_1^\# \leq f_2^\#$.
\end{prop}
\begin{proof}
If $f_1 \leq f_2$ then for fixed $s$, 
\[ \mathrm{length\ of\ } \{ f_1 \leq s \} \geq \mathrm{length\ of\ } \{ f_2 \leq s \}. \]
This implies for fixed $t$ that
\[ \left\{ s: t \leq \mathrm{length\ of\ } \{ f_1 \leq s\} \right\} \supset 
   \left\{ s: t \leq \mathrm{length\ of\ } \{ f_2 \leq s\} \right\}, \]
hence $f_1^\#(t) \leq f_2^\#(t)$.
\end{proof}

\begin{prop}
Let $f: [a,b] \rightarrow \mathbb{R}$ be a continuous function. Then the monotone reordering 
$f^\#:(0, b-a] \rightarrow \mathbb{R}$ is also continuous.
\end{prop}
\begin{proof}
We already know from Prop.~\ref{mon-re-inc-prop} that $f^\#$ is lower semi-continuous. Let $t \in (0, b-a]$ and
set $s := f^\#(t)$. For $t = b-a$ or $s = \max\{ f(t^\prime): t^\prime \in [a,b]\}$ nothing is to prove. Since $f$ is 
continuous the set $\{t^\prime: s < f(t^\prime) < s+\epsilon \}$ is open and non-empty for all $\epsilon > 0$ and has 
consequently a positive length $\delta_\epsilon$. Then 
\[ \mathrm{length\ of\ } \{ f \leq s+\epsilon \} \geq 
   \mathrm{length\ of\ } \{ f \leq s\} + \mathrm{length\ of\ } \{ s < f < s+\epsilon \} \geq t+\delta_\epsilon, \]
hence $f^\#(t^\prime) < s + 2\epsilon$ for all $t^\prime < t+\delta_\epsilon$, and $f^\#$ is upper 
semi-continuous in $t$.
\end{proof}

\begin{thm}  \label{cutoff-thm}
Let $f: [a,b] \rightarrow \mathbb{R}$ be a continuous function on the closed interval $[a,b] \subset \mathbb{R}$.
Then for all $\epsilon > 0$ there exists a $\delta > 0$ such that for all closed intervals 
$[a^\prime, b^\prime] \subset [a,b]$ with $(b-a) - (b^\prime-a^\prime) < \delta$ the monotone reorderings
$f^\#: (0, b-a] \rightarrow \mathbb{R}$ and 
$(f_{|[a^\prime,b^\prime]})^\#: (0, b^\prime-a^\prime] \rightarrow \mathbb{R}$ satisfy
\[ \parallel \!\! (f^\#)_{|(0, b^\prime-a^\prime]} - (f_{|[a^\prime,b^\prime]})^\# \!\! \parallel_{\mathrm{max}} < 
   \epsilon. \]
\end{thm}
\begin{proof}
Since $f$ is continuous on the compact interval $[a,b]$ the function is uniformly continuous on $[a,b]$.
Consequently, for every $\epsilon > 0$ there exists a $\delta > 0$ such that $| t - t^\prime | < \delta$ implies
$| f(t) - f(t^\prime) | < \epsilon$.

\vspace{0.2cm}

\noindent \textit{Claim 1}. $| t - t^\prime | < \delta$ also implies $| f^\#(t) - f^\#(t^\prime) | < 2\epsilon$.
\begin{proof}
Since $f$ is continuous $f$ has a minimum $s_{\min}$ on $[a,b]$ which also must be a lower bound for $f^\#$ on $(0,b-a]$.
By construction,
$\mathrm{length\ of\ } \{ f \leq s_{\min}+\epsilon\} \geq \delta$, hence
\[ | f^\#(t) - f^\#(t^\prime) | \leq \epsilon < 2\epsilon\ 
                                                       \mathrm{for\ all\ } 0 < t \leq t^\prime < \delta. \]

\noindent Now let $t^\prime \geq \delta$. If $f^\#(t^\prime) - 2\epsilon < s_{\min}$, we have 
\[ | f^\#(t^\prime) - f^\#(t) | \leq f^\#(t^\prime) - s_{\min} < 2\epsilon\ \mathrm{for\ } t < t^\prime, \]
since $f^\#$ is monotonely increasing by Prop.~\ref{mon-re-inc-prop}.

\noindent Otherwise $f^\#(t^\prime) - 2\epsilon \geq s_{\min}$, and we show two claims:

\vspace{0.2cm}

\noindent \textit{Claim 1.1}. $f: [a,b] \rightarrow \mathbb{R}$ continuous $\Rightarrow$ 
                                             $\mathrm{length\ of\ } \{ f < f^\#(t^\prime) \} \leq t^\prime$.
\begin{proof}
The characteristic function of the sets $\{ s-\epsilon < f < s \}$ tend pointwise to $0$ for $\epsilon \rightarrow 0$,
and they are dominated by the integrable characteristic function of $[a,b]$. Hence, by Lebesgue's dominated 
convergence,
\[ \mathrm{length\ of\ } \{ s-\epsilon < f \leq s \} \rightarrow 0,\ \mathrm{for\ } \epsilon \rightarrow 0. \]
If $\mathrm{length\ of\ } \{ f < f^\#(t^\prime) \} > t^\prime$, this limit would 
imply the existence of an $\epsilon>0$ such that
\[ \mathrm{length\ of\ } \{ f < f^\#(t^\prime)-\epsilon \} > t^\prime, \]
a contradiction to the definition of $f^\#(t^\prime)$.
\end{proof}

\vspace{0.2cm}

\noindent \textit{Claim 1.2}. $\mathrm{Length\ of\ } \{ f^\#(t^\prime) - 2\epsilon < f < f^\#(t^\prime) \} \geq \delta$.
\begin{proof}
$f^\#(t^\prime)$ is a value of $f$ on $[a,b]$: Otherwise $f^\#(t^\prime)$ would be bigger than the maximum 
$s_{\max}$ of $f$ on $[a,b]$, hence 
\[ \mathrm{length\ of\ } \{ f \leq f^\#(t^\prime) \} = \mathrm{length\ of\ } \{ f \leq s_{\max} \}, \] 
contradicting the definition of $f^\#(t^\prime)$. Furthermore, $f^\#(t^\prime) - 2\epsilon > s_{\min}$, hence by 
continuity, there exists a $\overline{t}$ such that 
$f(\overline{t}) = f^\#(t^\prime) - \epsilon$. But then the construction of $\delta$ shows that 
\[ f((\overline{t}-\delta, \overline{t}+\delta)) \subset (f^\#(t^\prime) - 2\epsilon,  f^\#(t^\prime)). \]
Since w.l.o.g. we can assume that $\delta < b-a$, the interval 
$(\overline{t}-\delta, \overline{t}+\delta) \cap [a,b]$ has length $\geq \delta$, hence the claim.
\end{proof}

\noindent From these two claims we deduce
\begin{eqnarray*}
\lefteqn{\mathrm{length\ of\ } \{ f \leq f^\#(t^\prime) - 2\epsilon \} = } \\
 & = & \mathrm{length\ of\ } \{ f < f^\#(t^\prime) \} - 
           \mathrm{length\ of\ } \{ f^\#(t^\prime) - 2\epsilon < f < f^\#(t^\prime) \} \\
 & \leq & t^\prime - \delta,
\end{eqnarray*}
hence we have $f^\#(t^\prime) \geq f^\#(t) > f^\#(t^\prime) - 2\epsilon$ for all $t^\prime - \delta < t \leq t^\prime$. 
This proves Claim 1.
\end{proof}

\noindent Now choose $a^\prime, b^\prime \in [a,b]$ such that 
$d := (b-a) - (b^\prime - a^\prime) < \delta$. Since $f$ is continuous it has a maximum $M$
and a minimum $m$ on the compact set $[a^\prime, b^\prime]$. Define
\[ \underline{f}(t) := 
   \left\{ \begin{array}{ll}
           m-\epsilon & \mathrm{for\ } t \in [a,a^\prime) \cup (b^\prime,b] \\
           f(t) & \mathrm{else}                    
           \end{array} \right. ,\ \ 
   \overline{f}(t) := 
   \left\{ \begin{array}{ll}
           M+\epsilon & \mathrm{for\ } t \in [a,a^\prime) \cup (b^\prime,b] \\
           f(t) & \mathrm{else.}                    
\end{array} \right.  \]
Then $\underline{f} \leq f \leq \overline{f}$, hence $\underline{f}^\# \leq f^\# \leq \overline{f}^\#$ by 
Prop.~\ref{monre-ineq-prop}, and furthermore
\[ \underline{f}^\#(t) = 
   \left\{ \begin{array}{ll}
           m-\epsilon & \mathrm{for\ } t \leq d \\
           (f_{|[a^\prime,b^\prime]})^\#(t-d) & \mathrm{else}                    
           \end{array} \right. ,\ \ 
   \overline{f}^\#(t) = 
   \left\{ \begin{array}{ll}
           (f_{|[a^\prime,b^\prime]})^\#(t) & \mathrm{for\ } t \leq b^\prime - a^\prime \\
           M+\epsilon & \mathrm{else.}                    
\end{array} \right.  \]
Consequently, for $t \leq b^\prime - a^\prime$,
\begin{eqnarray*}
| (f_{|[a^\prime,b^\prime]})^\#(t) - f^\#(t) | & \leq & | \overline{f}^\#(t) - \underline{f}^\#(t) | \\
    & = & \left\{ \begin{array}{ll}
           (f_{|[a^\prime,b^\prime]})^\#(t) - (m-\epsilon) & \mathrm{for\ } t \leq d \\
           (f_{|[a^\prime,b^\prime]})^\#(t) - (f_{|[a^\prime,b^\prime]})^\#(t-d) & \mathrm{else.}                    
\end{array} \right. \\
    & < & \left\{ \begin{array}{c}
                  3\epsilon\\
                  2\epsilon
                  \end{array} \right.
\end{eqnarray*}
using $d < \delta$, $m \leq (f_{|[a^\prime,b^\prime]})^\#(t)$ for all $t \in (0,b^\prime - a^\prime]$ and Claim~1 applied 
to $f_{|[a^\prime,b^\prime]}$.
\end{proof}

\begin{prop} \label{max-norm-bd-prop} 
Let $f,g: [a,b] \rightarrow \mathbb{R}$ be two continuous function on $[a,b]$. Then for all $\epsilon > 0$,
\[ \parallel\!\! f - g \!\!\parallel_{\max} < \epsilon\ \ \Longrightarrow \ \  
   \parallel\!\! f^\# - g^\# \!\!\parallel_{\max} < 2\epsilon. \]
\end{prop}
\begin{proof}
$\parallel\!\! f - g \!\!\parallel_{\max} < \epsilon$ implies $f-\epsilon < g < f+\epsilon$. Since 
$(f \pm \epsilon)^\# = f^\# \pm \epsilon$, we obtain from Prop.~\ref{monre-ineq-prop}
\[ f^\# - \epsilon \leq g^\# \leq f^\# + \epsilon, \]
hence the claim.
\end{proof}

\begin{prop} \label{mon-re-id-prop}
Let $f: [a,b] \rightarrow \mathbb{R}_{\geq 0}$ be a continuous concave function, 
$f^\#: (0,b-a] \rightarrow \mathbb{R}_{\geq 0}$ its monotone reordering and $M := \max \{ f(t): t \in [a,b] \}$.
If $M \geq b-a$ then $f^\#(t) \geq t$ for all $t \in (0, b-a]$.
\end{prop}
\begin{proof}
Since $f$ is continuous, it achieves its maximum in some point $c \in [a,b]$. From now on suppose $c \in (a,b)$.
If $c=a$ or $c=b$ the arguments are similar but easier. Set
\[ g: [a,b] \rightarrow \mathbb{R}_{\geq 0}, 
   t \mapsto \left\{ \begin{array}{ll}
                            \frac{t-a}{c-a} \cdot (b-a)\ \mathrm{for\ all\ } a \leq t \leq c \\
                            \frac{b-t}{b-c} \cdot (b-a)\ \mathrm{for\ all\ } c \leq t \leq b. \rule{0cm}{0.4cm}
                            \end{array} \right. \]
Since $\mathrm{length\ of\ } \{ g \leq s \} = s$, we have $g^\#(t) = t$ for all $t \in (0,b-a]$. On the other hand,
$g \leq f$ because $M \geq b-a$ and $f$ is concave. Consequently,
\[ g^\# \leq f^\# \]
by Prop.~\ref{monre-ineq-prop}, which proves the claim.
\end{proof}

\section{The asymptotic version of Dumnicki's algorithm}

\noindent The asymptotic version of Dumnicki's reduction algorithm now reads as follows:
\begin{thm}[Asymptotic version of Dumnicki's reduction algorithm]  \label{As-red-alg-Thm}
Let $m_1, \ldots. m_{p-1}, m_p \in \mathbb{R}_{> 0}$ and $P \subset \mathbb{R}^2_{\geq 0}$. For
\[ F: \mathbb{R}^2_{\geq 0} \ni (\alpha_1, \alpha_2) \mapsto r_0 + r_1 \alpha_1 + r_2 \alpha_2,\ 
   r_0, r_1, r_2 \in \mathbb{R}, \]
an affine function, define
\begin{eqnarray*}
P_1 & := & P \cap \{ (\alpha_1, \alpha_2) : F(\alpha_1, \alpha_2) < 0\} \\
P_2 & := & P \cap \{ (\alpha_1, \alpha_2) : F(\alpha_1, \alpha_2) > 0\}. 
\end{eqnarray*}
If $P_1$ contains asymptotically $(m_p)$-non-special systems of dimension $-1$ and
$P_2$ contains asymptotically $(m_1, \ldots, m_{p-1})$-non-special systems of dimension $\geq 0$, then
$P$ contains asymptotically $(m_1, \ldots , m_p)$-non-special systems of dimension $\geq 0$.
\end{thm}
\begin{proof}
By assumption, the set $n \cdot P_1 \cap \mathbb{N}^2_{\geq 0}$ contains an $m_p^{(n)}$-non-special system of dimension
$-1$, and $n \cdot P_2 \cap \mathbb{N}^2_{\geq 0}$ contains an $(m_1^{(n)}, \ldots, m_{p-1}^{(n)})$-non-special system of 
dimension $\geq 0$, such that the $m_i^{(n)}$ satisfy the inequalities $(ii)$ of Def.~\ref{asymp-sys-Def} for a given
$\delta$, for all $n \gg 0$.

\noindent Consequently, Dumnicki's reduction algorithm applied to $n \cdot P \cap \mathbb{N}^2_{\geq 0}$ and 
\[ F_n: \mathbb{R}^2_{\geq 0} \ni (\alpha_1, \alpha_2) \mapsto nr_0 + r_1 \alpha_1 + r_2 \alpha_2,\]
shows that $n \cdot P \cap \mathbb{N}^2_{\geq 0}$ contains an $(m_1^{(n)}, \ldots, m_p^{(n)})$-non-special system of 
dimension $\geq 0$. Since the $m_i^{(n)}$ still satisfy the inequalities $(ii)$ of Def.~\ref{asymp-sys-Def}, for given 
$\delta$, the algorithm is justified.
\end{proof}

\noindent The facts shown in the last section can be used to prove a criterion for asymptotic $(m)$-non-specialty:
\begin{thm}[Criterion for asymptotic $(m)$-non-specialty] \label{as-m-non-spec-crit}
Let $P$ be a convex open subset of $\mathbb{R}^2_{\geq 0}$, such that its closure $\overline{P}$ is compact. Set 
$[a,b] := p_x(\overline{P})$ where $p_x: \mathbb{R}^2_{\geq 0} \rightarrow \mathbb{R}_{\geq 0}$ is the 
projection of $\mathbb{R}^2_{\geq 0}$ onto the positive $x$-axis. Define
\[ f: [a,b] \rightarrow \mathbb{R}_{\leq 0}, t \mapsto f(t) := \mathrm{length\ of\ } p_x^{-1}(t) \cap \overline{P}, \]
and let $f^\#: (0, b-a] \rightarrow \mathbb{R}_{\leq 0}$ be the monotone reordering of $f$.

\noindent If $m < b-a$ and $f^\#(t) \geq t$, for all $t \in (0,m]$, then $P$ contains asymptotically $(m)$-non-special
systems of dimension $-1$. 
\end{thm}

\noindent The following diagram illustrates how we obtain the height function$f$ from a convex polygon 
$P \subset \mathbb{R}^2_{\geq 0}$:

\begin{center}
\begin{picture}(180,80)(-5,-5)
\put(-5,0){\vector(1,0){85}}
\put(0,-5){\vector(0,1){65}}

\put(0,30){\line(2,1){40}}
\put(40,50){\line(1,-1){20}}
\put(60,30){\line(-1,-2){10}}
\put(50,10){\line(-1,0){30}}
\put(20,10){\line(-1,1){20}}

\put(30,25){\small P}

\put(20,-1){\line(0,1){2}}
\put(40,-1){\line(0,1){2}}
\put(50,-1){\line(0,1){2}}
\put(60,-1){\line(0,1){2}}

\put(8,-8){\small $t_1$}
\put(28,-8){\small $t_2$}
\put(43,-8){\small $t_3$}
\put(53,-8){\small $t_4$}

\put(90,30){\vector(1,0){10}}

\put(110,0){\vector(1,0){85}}
\put(115,-5){\vector(0,1){65}}

\put(115,0){\line(2,3){20}}
\put(135,30){\line(2,1){20}}
\put(155,40){\line(1,-1){10}}
\put(165,30){\line(1,-3){10}}

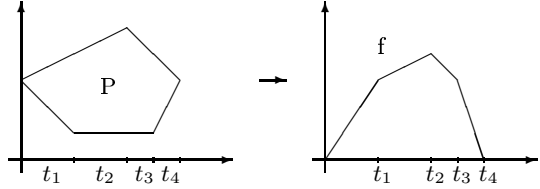
\put(135,40){\small f}

\put(135,-1){\line(0,1){2}}
\put(155,-1){\line(0,1){2}}
\put(165,-1){\line(0,1){2}}
\put(175,-1){\line(0,1){2}}

\put(133,-8){\small $t_1$}
\put(153,-8){\small $t_2$}
\put(163,-8){\small $t_3$}
\put(173,-8){\small $t_4$}
\end{picture}
\end{center}

\vspace{0.2cm}

\noindent For the proof of the theorem we need some further properties of the function $f$:
\begin{prop}  \label{conc-cont-prop}
Let $P$ be a convex open subset of $\mathbb{R}^2_{\geq 0}$, $\overline{P}$ compact. Define 
$f: [a,b] \rightarrow \mathbb{R}_{\leq 0}$ as in Thm.~\ref{as-m-non-spec-crit}. Then $f$ is concave and continuous on
$[a,b]$.
\end{prop}
\begin{proof}
Let $f^+$ resp. $f^-$ be the functions assigning to each $t \in [a,b]$ the upper resp. lower bound of the interval
$p_x^{-1}(t) \cap P$. Then $f(t) = f^+(t) - f^-(t)$.

\noindent $f^+$ is concave: Choose $t_1, t_2 \in [a,b]$. Then the convexity of $\overline{P}$ implies
\[ \left( \lambda t_1 + (1-\lambda) t_2, \lambda f^+(t_1) + (1-\lambda)f^+(t_2) \right) \in \overline{P}, \]
hence
\[ \lambda f^+(t_1) + (1-\lambda)f^+(t_2) \leq f^+(\lambda t_1 + (1-\lambda) t_2). \]
Similarly $f^-$ is convex, hence $f$ as the difference of a concave and a convex function is concave.

\noindent The following lemma shows that $f$ is also continuous on $(a,b)$ and that 
$\lim_{t \rightarrow a} f^\pm(t)$ and $\lim_{t \rightarrow b} f^\pm(t)$ exist.

\noindent We still have to prove that $\lim_{t \rightarrow a} f^\pm(t) = f^\pm(a)$ resp.  
$\lim_{t \rightarrow b} f^\pm(t) = f^\pm(b)$: Closedness implies $(a, \lim_{t \rightarrow a} f^+(t)) \in \overline{P}$,
hence $\lim_{t \rightarrow a} f^+(t) \leq f^+(a)$. If $\lim_{t \rightarrow a} f^+(t) < f^+(a)$ then
\[ \lambda f^+(a) + (1-\lambda) f^+(t) > (1-\lambda) f^+(t) \]
for $\lambda$ sufficiently close to $1$, contradicting the concavity of $f^+$. The same types of arguments hold for 
the other limits. 
\end{proof}

\begin{lem}
Let $f: [a,b] \rightarrow \mathbb{R}$ be a concave function. Then $f$ is continuous on $(a,b)$, and the limits
$\lim_{t \rightarrow a} f(t)$ and $\lim_{t \rightarrow b} f(t)$ exist.
\end{lem}
\begin{proof}
For any triple $a \leq t^\prime < t_0 < t^{\prime\prime} \leq b$, we have 
$t_0 = \frac{t^{\prime\prime} - t_0}{t^{\prime\prime} - t^\prime} t^\prime +
           \frac{t_0 - t^\prime}{t^{\prime\prime} - t^\prime} t^{\prime\prime}$. The concavity of $f$ implies 
\[ f(t_0) \geq \frac{t^{\prime\prime} - t_0}{t^{\prime\prime} - t^\prime} f(t^\prime) +
           \frac{t_0 - t^\prime}{t^{\prime\prime} - t^\prime} f(t^{\prime\prime}). \]
Subtracting $f(t^\prime)$ from both sides leads to the left hand inequality of
\[ \frac{f(t_0)-f(t^\prime)}{t_0-t^\prime} \geq \frac{f(t^{\prime\prime})-f(t^\prime)}{t^{\prime\prime}-t^\prime} \geq 
   \frac{f(t^{\prime\prime})-f(t_0)}{t^{\prime\prime}-t_0}, \]
subtracting $f(t^{\prime\prime})$ and multiplying with $-1$ to the right hand inequality. 
Renaming $t_0,t^{\prime}, t^{\prime\prime}$ the left hand inequality implies that  the difference
quotients $\frac{f(t^{\prime\prime})-f(t_0)}{t^{\prime\prime}-t_0}$ are increasing for 
$t^{\prime\prime} \stackrel{>}{\rightarrow} t_0$, so they are bounded from below for $t^{\prime\prime}$ close to 
$t_0$. Since $t_0 \in (a,b)$ there is always a $t^\prime < t_0$ in $(a,b)$, hence the inequality chain shows that the 
difference quotients are also bounded from above. Consequently,
\[ | f(t^{\prime\prime})-f(t_0) | < M \cdot | t^{\prime\prime}-t_0 | \]
for appropriate $M > 0$ and $t^{\prime\prime} > t_0$ close to $t_0$. The same can be shown for $t^\prime < t_0$
close to $t_0$, hence $f$ is continuous in $t_0$.

\noindent The situation is more complicated for the boundary points $a$ and $b$. If 
$\lim_{t \stackrel{<}{\rightarrow} b} f(t)$ does not exist there are $3$ possibilities:
\begin{itemize}
\item[(1)] There are at least $2$ accumulation points, for different sequences $(t_n^\pm) \stackrel{<}{\rightarrow} b$,
say $-\infty \leq y^- < y^+ \leq +\infty$. Then it is possible to choose a triple $t_0 < t_n^- < t_m^+$, $n,m \gg 0$, which
contradicts the concavity of $f$ (see the end of the proof in the proposition before).
\item[(2)] $\lim_{t \stackrel{<}{\rightarrow} b} f(t) = -\infty$: Then it is possible to choose a triple $t_0 < t < b$ 
contradicting the concavity of $f$ as before.
\item[(3)] $\lim_{t \stackrel{<}{\rightarrow} b} f(t) = +\infty$: Then it is possible to choose 
$t_0 < t^\prime < t^{\prime\prime} < b$, such that the point $(t^{\prime\prime}, f(t^{\prime\prime}))$ lies over
the line connecting $(t_0, f(t_0))$ and $(t^\prime, f(t^\prime))$. But then the point $(t^\prime, f(t^\prime))$ lies under the
line segment connecting $(t_0, f(t_0))$ and $(t^{\prime\prime}, f(t^{\prime\prime}))$. This contradicts the concavity of 
$f$.

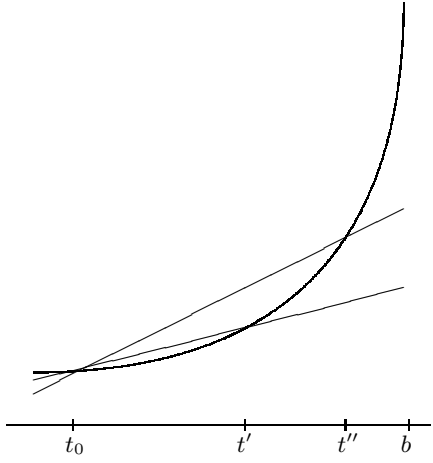
\begin{center}
\begin{picture}(200,180)(-20,-20)
\qbezier(0,20)(140,20)(140,160)
\put(0,12){\line(2,1){140}}
\put(0,17){\line(4,1){140}}

\put(-10,0){\line(1,0){160}}

\put(15,-2){\line(0,1){4}}
\put(80,-2){\line(0,1){4}}
\put(118,-2){\line(0,1){4}}
\put(142,-2){\line(0,1){4}}

\put(12,-10){\small $t_0$}
\put(77,-10){\small $t^\prime$}
\put(115,-10){\small $t^{\prime\prime}$}
\put(139,-10){\small $b$}
\end{picture}
\end{center}
\end{itemize}
The same type of arguments hold for $\lim_{t \stackrel{>}{\rightarrow} a} f(t)$.
\end{proof}

\noindent \textit{Proof of Thm.~\ref{as-m-non-spec-crit}.} We start with a construction: Take any number 
$n \in \mathbb{N}$. Set 
\[ S_{(m_x,m_y)} := \left[ \frac{m_x}{n} - \frac{1}{2n}, \frac{m_x}{n} + \frac{1}{2n} \right] \times
                    \left[ \frac{m_y}{n} - \frac{1}{2n}, \frac{m_x}{n} + \frac{1}{2n} \right], 
   (m_x, m_y) \in \mathbb{N}^2. \]
The $S_{(m_x,m_y)}$ are closed squares of sidelength $\frac{1}{n}$ such that the coordinates of their centers are 
multiples of $\frac{1}{n}$. Then set $P_n := \bigcup_{S_{(m_x,m_y)} \subset \overline{P}} S_{(m_x,m_y)}$, the union of 
all such squares in $\overline{P}$.

\noindent Since $P_n \subset \overline{P}$ we have $p_x(P_n) \subset [a,b]$. Hence we can define the analog of $f$ for
$P_n$:
\[ f_n: [a,b] \rightarrow \mathbb{R}_{\geq 0},\ t \mapsto \mathrm{length\ of\ } p_x^{-1}(t) \cap P_n. \]
By construction, $f_n \leq f$. Furthermore the following holds: 

\vspace{0.2cm}

\noindent \textit{Claim 1.} For all $a < a_0 < b_0 < b$, the functions $f_n$ converge uniformly against $f$ on 
$[a_0,b_0]$, for $n \rightarrow \infty$. 
\begin{proof}
Let $f^+$ and $f^-$ be defined as in the proof of Prop.~\ref{conc-cont-prop}. This proof shows that $f^+$ and $f^-$ are
continuous functions, hence they are uniformly continuous on the compact interval $[a,b]$. Consequently, for every 
$\epsilon > 0$ there exists a $\delta > 0$ such that 
\[ |x-y| < \delta \Rightarrow |f^\pm(x) - f^\pm(y)| < \epsilon\ \mathrm{for\ all\ } x,y \in [a,b].\]
Next, choose $a < a^\prime < a_0 < b_0 < b^\prime < b$. Since $P$ is a convex set and the projection $p_x$ is an open and
continuous map, 
$p_x(P) = (a,b)$, and $p_x^{-1}(t) \cap P$ is a non-empty open interval, for all $t \in (a,b)$. Consequently, $f > 0$ on
$(a,b)$, and $f$ achieves a strictly positive minimum on $[a^\prime, b^\prime]$.

\noindent Let $\epsilon > 0$ be any real number such that $4\epsilon$ is smaller than this minimum. Let 
$n \in \mathbb{N}$ be an integer such that $\frac{1}{n} < \epsilon$, $\frac{1}{n} < \delta$, 
$\frac{1}{n} < \min \{a_0-a^\prime, b^\prime-b_0, b_0-a_0\}$. Let $\frac{k}{n} \in [a_0,b_0]$. By assumption,
\[ f^+(k/n) - f^-(k/n) > 4\epsilon > \epsilon + \frac{1}{2n} + \frac{1}{n} + \frac{1}{2n} + \epsilon. \]
Hence the interval $[f^-(k/n) + \epsilon + \frac{1}{2n}, f^+(k/n) - \epsilon - \frac{1}{2n}]$ has length $> \frac{1}{n}$,
consequently it contains at least one number of the form $\frac{m}{n}$.

\vspace{0.2cm}

\noindent \textit{Claim 1.1.} 
$[\frac{k}{n}-\frac{1}{2n}, \frac{k}{n}+\frac{1}{2n}] \times [\frac{m}{n}-\frac{1}{2n}, \frac{m}{n}+\frac{1}{2n}]
  \subset \overline{P}$.
\begin{proof}
Let $t \in [\frac{k}{n}-\frac{1}{2n}, \frac{k}{n}+\frac{1}{2n}] \subset [a^\prime, b^\prime]$. Then 
$|t-\frac{k}{n}|<\delta$, hence $|f^\pm(t)-f^\pm(k/n)|<\epsilon$. This implies 
$f^+(t) > f^+(k/n)-\epsilon \geq \frac{m}{n}+\frac{1}{2n}$ and 
$\frac{m}{n}-\frac{1}{2n} \geq f^-(k/n)+\epsilon > f^-(t)$.
\end{proof}

\noindent If $\frac{L}{n}$ is maximal among all 
$\frac{m}{n} \in [f^-(k/n) + \epsilon + \frac{1}{2n}, f^+(k/n) - \epsilon - \frac{1}{2n}]$, then 
$\frac{L}{n} \geq f^+(k/n) - \epsilon - \frac{1}{2n} - \frac{1}{n}$, and if $\frac{l}{n}$ is minimal, then
$\frac{l}{n} \leq f^-(k/n) + \epsilon + \frac{1}{2n} + \frac{1}{n}$. Consequently
\[ f^+(t) - (\frac{L}{n} + \frac{1}{2n}) < f^+(t) - (f^+(k/n) - \epsilon - \frac{1}{n}) \leq 
   |f^+(t) - f^+(k/n)| + \epsilon + \frac{1}{n} \leq 2\epsilon + \frac{1}{n}. \]
Similarly we deduce $\frac{l}{n} - \frac{1}{2n} - f^-(t) < 2\epsilon + \frac{1}{n}$.

\noindent Now, $f(t) = f^+(t) - f^-(t) \geq f_n(t) > \frac{L}{n} - \frac{l}{n}$ for 
$t \in [\frac{k}{n}-\frac{1}{2n}, \frac{k}{n}+\frac{1}{2n}]$, hence 
\[ f(t) - f_n(t) \leq f^+(t) - \frac{L}{n} + \frac{l}{n} - f^-(t) < 4\epsilon + \frac{3}{n} < 7\epsilon. \]
Claim 1 is proven.
\end{proof}

\vspace{0.2cm}

\noindent \textit{Claim 2.} For every $\epsilon > 0$ and $n \gg 0$,
\[ \parallel\!\! (f^\#)_{|(0,m]} - (f_n^\#)_{|(0,m]} \!\!\parallel_{\max} < \epsilon. \] 
\begin{proof}
Given $\epsilon > 0$, there is a $\delta > 0$ such that
\[ \parallel\!\! (f^\#)_{|(0,b_0-a_0]} - (f_{|[a_0,b_0]})^\# \!\!\parallel_{\max} < \frac{\epsilon}{3} \]
as long as $(a_0-a)+(b-b_0) < \delta$, by means of Thm.~\ref{cutoff-thm}. Furthermore, there is an $n \gg 0$ such that
$\parallel\!\! f - f_n \!\!\parallel_{\max} < \frac{\epsilon}{6}$, by Claim 1, hence also
$\parallel\!\! f_{|[a_0,b_0]} - f_{n|[a_0,b_0]} \!\!\parallel_{\max} < \frac{\epsilon}{6}$. Then 
Prop.~\ref{max-norm-bd-prop} implies
\[ \parallel\!\! (f_{|[a_0,b_0]})^\# - (f_{n|[a_0,b_0]})^\# \!\!\parallel_{\max} < \frac{\epsilon}{3}. \]
Next we choose $\delta$ small enough to ensure 
$\parallel\!\! (f_{n|[a_0,b_0]})^\# - (f_n^\#)_{|(0,b_0-a_0]} \!\!\parallel_{\max} < \frac{\epsilon}{3}$; this is 
possible again by Thm.~\ref{cutoff-thm}.

\noindent Finally we choose $a_0, b_0$ such that in addition to $(a_0-a)+(b-b_0) < \delta$, we also have 
$m < b_0-a_0 < b-a$. By expanding $(f^\#)_{|(0,m]} - (f_n^\#)_{|(0,m]}$ to
\[ (f^\#)_{|(0,m]} - (f_{|[a_0,b_0]})^\#_{|(0,m]} + (f_{|[a_0,b_0]})^\#_{|(0,m]} - (f_{n|[a_0,b_0]})^\#_{|(0,m]} + 
   (f_{n|[a_0,b_0]})^\#_{|(0,m]} -(f_n^\#)_{|(0,m]} \]
we get the claim.
\end{proof}

\noindent For all $\epsilon > 0$, Claim 2 implies $f_n^\#(t) > t-\epsilon$ for all $t \in [0,m]$ if $n \gg 0$. 

\noindent Now for $\epsilon$ small enough, consider all integers $e$ with 
$0 \leq e \leq \lfloor mn \rfloor - \lceil n\epsilon \rceil - 1$, and set
$e^\prime := \lfloor mn \rfloor - \lceil n\epsilon \rceil - e - 1$. Then 
$e + 1 + \lceil n\epsilon \rceil \leq \lfloor mn \rfloor$, hence $\frac{e + 1 + \lceil n\epsilon \rceil}{n} \leq m$, and 
we can apply $f_n^\#$ on $\frac{e + 1 + \lceil n\epsilon \rceil}{n}$:
\[ f_n^\#(\frac{e + 1 + \lceil n\epsilon \rceil}{n}) = f_n^\#(\frac{\lfloor mn \rfloor - e^\prime}{n}) > 
   \frac{\lfloor mn \rfloor - e^\prime}{n} - \epsilon \geq 
   \frac{\lfloor mn \rfloor - e^\prime - \lceil n\epsilon \rceil}{n} = \frac{e+1}{n}, \]
for $n \gg 0$. Consequently for each such e the step function $f_n^\#$ has a step of height at least $\frac{e+1}{n}$, 
and these steps can be chosen pairwise distinct.

\noindent On the other hand, $f_n^\#$ is just a reordering of the steps in $f_n$, so $f_n$ has the same property. By 
construction the height of the step of $f_n$ over $\frac{k}{n}$ counts the number of points $(\frac{k}{n}, \frac{l}{n})$
inside $P$. But this means that $n \cdot P \cap \mathbb{N}^2_{\geq 0}$ contains a non-special linear system 
$\mathcal{L}_{D_n}(\lfloor mn \rfloor - \lceil n\epsilon \rceil)$ of dimension $-1$, by Dumnicki's non-specialty 
criterion. Since 
\[ \left| \frac{\lfloor mn \rfloor - \lceil n\epsilon \rceil - nm}{nm} \right| \leq 
   \left| \frac{\lfloor mn \rfloor - nm}{nm} \right| + \frac{\lceil n\epsilon \rceil}{nm} \leq 
   \frac{1}{mn} + \frac{\epsilon}{m} + \frac{1}{nm} \rightarrow 0 \]
for $\epsilon \rightarrow 0$, $n \rightarrow \infty$, the theorem is proven.
 \hfill $\Box$

\section{A lower bound for the Seshadri constant of 10 points in $\mathbb{CP}^2$} \label{LowerBound-sec}

\noindent The following proposition allows to prove the nefness criterion Prop.~\ref{nef-crit}:
\begin{prop} \label{inc-non-spc-prop}
Let $m_1, \ldots, m_r \in \mathbb{N}_{>0}$, let $D_1 \subset D_2 \subset \mathbb{N}^2$ be finite subsets, and let
$p_1, \ldots, p_r \in \mathbb{P}^2$ be points. If $\mathcal{L}_{D_1}(m_1p_1, \ldots, m_rp_r)$ has expected dimension 
$\geq 0$, then also $\mathcal{L}_{D_2}(m_1p_1, \ldots, m_rp_r)$.
\end{prop}
\begin{proof}
The subspace $\mathcal{L}_{D_1}(m_1p_1, \ldots, m_rp_r) \subset \mathcal{L}_{D_2}(m_1p_1, \ldots, m_rp_r)$ is described 
by the intersection of $\mathcal{L}_{D_2}(m_1p_1, \ldots, m_rp_r) \subset \mathbb{P}(\sum_{i+j \leq d} a_{ij}x^iy^j)$ 
with the linear subspace
\[ \{ a_{ij} = 0: (i,j) \in D_2 \setminus D_1 \}. \]
Here, $d$ is chosen such that $D_1 \subset D_2 \subset \{ (i,j) : i+j \leq d \}$. Consequently, 
\[ \dim \mathcal{L}_{D_1}(m_1p_1, \ldots, m_rp_r) \geq \mathcal{L}_{D_2}(m_1p_1, \ldots, m_rp_r) - 
                                                       \#(D_2 \setminus D_1). \]
In particular, $\mathcal{L}_{D_1}(m_1p_1, \ldots, m_rp_r)$ cannot have expected dimension if 
$\mathcal{L}_{D_2}(m_1p_1, \ldots, m_rp_r)$ has not.
\end{proof}

\noindent \textit{Proof of Prop.~\ref{nef-crit}.} The assumption on $P$ implies that there exists a sequence 
$\delta_n \stackrel{>}{\rightarrow}0$, natural numbers $d_n, m_1^{(n)}, \ldots, m_r^{(n)}$ for all $n \in \mathbb{N}$, 
$d_n \rightarrow \infty$ for $n \rightarrow \infty$, and subsets $D_n \subset d_n \cdot P \cap \mathbb{N}^2_{\geq 0}$
such that $\mathcal{L}_{D_n}(m_1^{(n)}, \ldots, m_r^{(n)})$ is non-special of dimension $\geq 0$ and
\[ \left| \frac{m_i^{(n)}}{d_n \cdot m} - 1 \right| = \left| \frac{m_i^{(n)}- d_n m}{d_n m} \right| < \delta_n,\ \  
   i = 1, \ldots, r. \]
Let $m_n := \min \{ m_1^{(n)}, \ldots, m_r^{(n)} \}$. Then Prop.~\ref{inc-non-spc-prop} implies that 
$\mathcal{L}_{d_n}(m_n^r)$ is non-special of dimension $\geq 0$.

\noindent Since $\frac{d_n}{m_n-1}$ and $\frac{d_n}{m_n}$ have the same limit, and for some $i_n \in \{ 1, \ldots, r\}$,
\[ \frac{d_n}{m_n} = \frac{d_n}{m_{i_n}^{(n)}} = \frac{d_n m}{m_{i_n}^{(n)}} \cdot \frac{1}{m} \rightarrow \frac{1}{m},\]
the proposition follows.
\hfill $\Box$

\vspace{0.2cm}

\noindent \textit{Proof of Thm.~\ref{Seshadri-bd-thm}.} In the following diagram, let the points $O,A,B, \ldots , R, S$
be given by the coordinates
\[ \begin{array}{lllll}
   O = (0,0), & A = (1,0), & C = (\frac{9}{13}, 0), & E = (\frac{9}{13}, \frac{4}{13}), & G = (\frac{5}{13}, 0), \\ 
   & & & & \\
             & B = (0,1), & D = (0, \frac{9}{13}), & F = (\frac{4}{13}, \frac{9}{13}), & H = (0, \frac{5}{13}), \\
   & & & & \\
             & I = (\frac{4}{13}, 0), & K = (\frac{7}{13}, \frac{6}{13}), & M = (\frac{6}{13}, \frac{3}{13}), & \\
   & & & & \\
             & J = (0, \frac{4}{13}), & L = (\frac{6}{13}, \frac{7}{13}), & N = (\frac{3}{13}, \frac{6}{13}), & \\
   & & & & \\
             & P = (\frac{9}{26}, \frac{9}{26}), & Q = (\frac{2}{13}, \frac{2}{13}), & 
               R = (\frac{7}{13}, \frac{2}{13}), & S = (\frac{9}{26}, 0).
   \end{array} \]

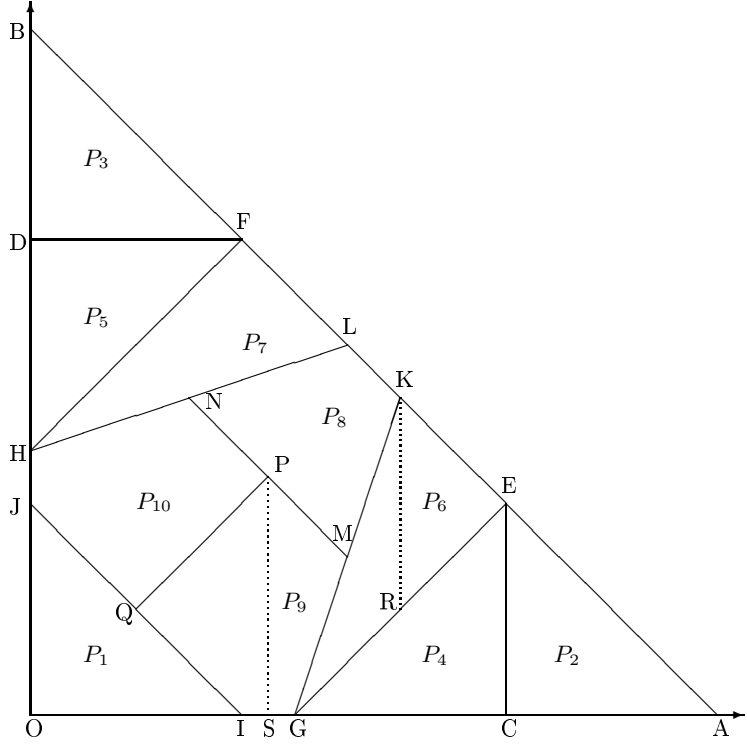
\begin{center}
\begin{picture}(280,280)(-10,-10)
\put(0,0){\vector(1,0){270}}
\put(0,0){\vector(0,1){270}}
\put(0,260){\line(1,-1){260}}

\put(0,80){\line(1,-1){80}}
\put(180,0){\line(0,1){80}}
\put(0,180){\line(1,0){80}}
\put(100,0){\line(1,1){80}}
\put(0,100){\line(1,1){80}}
\put(100,0){\line(1,3){40}}
\put(0,100){\line(3,1){120}}
\put(60,120){\line(1,-1){60}}
\put(40,40){\line(1,1){50}}
\qbezier[40](140,120)(140,80)(140,40)
\qbezier[40](90,90)(90,45)(90,0)

\put(-2,-8){\small O}
\put(258,-8){\small A}
\put(178,-8){\small C}
\put(98,-8){\small G}
\put(88,-8){\small S}
\put(78,-8){\small I}
\put(-8,256){\small B}
\put(-8,176){\small D}
\put(-8,96){\small H}
\put(-8,76){\small J}
\put(178,84){\small E}
\put(118,144){\small L}
\put(138,124){\small K}
\put(78,184){\small F}
\put(66,116){\small N}
\put(92,92){\small P}
\put(114,66){\small M}
\put(32,36){\small Q}
\put(132,40){\small R}

\put(20,20){\small $P_{1}$}
\put(198,20){\small $P_{2}$}
\put(20,208){\small $P_{3}$}
\put(148,20){\small $P_{4}$}
\put(20,148){\small $P_{5}$}
\put(148,78){\small $P_{6}$}
\put(80,138){\small $P_{7}$}
\put(110,110){\small $P_{8}$}
\put(95,40){\small $P_{9}$}
\put(40,78){\small $P_{10}$}

\end{picture}
\end{center}

\noindent
The diagram is possible since $E, K, L, F$ lie on the line $AB$, the points $I, G, C$ on the line $OA$, the points 
$J, H, D$ on the line $OB$, the point $M$ on the line $GK$, the point $N$ on the line $HL$, the point $P$ on the line 
$NM$ and $Q$ on the line $IJ$. Furthermore, $S$ lies between $I$ and $G$ on $OA$, and $R$ lies on the line $GE$.

\noindent The diagram shows the dissection of the $2$-dimensional simplex $OAB$ into 10 polygons $P_1, \ldots, P_{10}$
by straight lines. The indices of the polygons denote the sequence of dissections. To prove the theorem we apply the 
asymptotic version of Dumnicki's reduction algorithm to this sequence of dissections. That is, we have to show that for
every $m < \frac{4}{13}$ each of the polygons $P_i$, $i = 1, \ldots, 9$, contains asymptotically $(m)$-non-special 
systems of dimension $-1$, and that $P_{10}$ contains asymptotically $(m)$-non-special systems of dimension $\geq 0$.

\noindent By construction the polygons are convex. Hence Thm.~\ref{as-m-non-spec-crit} together with 
Prop.~\ref{mon-re-id-prop} and Prop.~\ref{conc-cont-prop} imply that it is enough to show the following, for every 
polygon $P_i$: The projection of $P_i$ onto the $x$-axis is an interval of length $\geq \frac{4}{13} > m$, and there is
a vertical section of $P_i$ of length $\geq \frac{4}{13} > m$. By symmetry, it is also possible to show these 
inequalities for the projection onto the $y$-axis and a horizontal section. Furthermore, since the lengths are $> m$, it 
will be always possible to add some monomials to $D_{10}^{(n)}$ in $P_{10}^{(n)}$. Prop.~\ref{inc-non-spc-prop} shows 
that this produces $m_{10}^{(n)}$-non-special systems of dimension $\geq 0$ in $P_{10}^{(n)}$, for $n \gg 0$. 

\noindent For the polygons $P_1, \ldots, P_5$ the inequalities are obvious. For the polygon $P_6$ the projection to the 
$x$-axis is the interval $GC$ which has length $\frac{4}{13}$. The vertical section $KR$ has also length $\frac{4}{13}$.
By symmetry, $P_7$ also satisfies the inequalities.

\noindent The projection of $P_8$ onto the $x$-axis is the interval $[\frac{3}{13}, \frac{7}{13}]$ (these are the 
$x$-coordinates of $N$ and $M$), and the vertical section $LM$ has length $\frac{4}{13}$. The projection of $P_9$ onto
the $x$-axis is $[\frac{2}{13}, \frac{6}{13}]$, and the vertical section $PS$ has length $\frac{4}{13}$. By symmetry,
$P_{10}$ also satisfies the necessary inequalities.
\hfill $\Box$

\begin{rem}
Since the bound $\frac{13}{4}$ is rational there might be a pair $(d,m)$ with $\mathcal{L}_d(10^{m+1})$ non-special of 
non-negative dimension and $\frac{d}{m} = \frac{13}{4}$. From such a pair the theorem would follow by Thm.~\ref{Eckl-thm}
without any limit process. But it is difficult to find such a pair: $\mathcal{L}_d(10^{m+1})$ has expected dimension $-1$
up to $m = 92$, and then it is still not clear how to prove non-specialty. For example, the cutting proposed in the proof
of the theorem above might require an even bigger $m$.
\end{rem}

%\bibliographystyle{alpha}

%\bibliography{/Forschung/Mathematik/doktor}
%\bibliography{doktor}

\begin{thebibliography}{Dum06}

\bibitem[Bir99]{Bir99}
Paul Biran.
\newblock Constructing new ample divisors out of old ones.
\newblock {\em Duke Math. J.}, 98(1):113--135, 1999.

\bibitem[CM01]{CilMir01}
C.~Ciliberto and R.~Miranda.
\newblock {The Segre and Harbourne--Hirschowitz conjectures}.
\newblock In {\em {Applications of algebraic geometry to coding theory, physics
  and computation (Eilat, 2001)}}, volume~36 of {\em NATO Sci. Ser. II Math.
  Phys. Chem.}, pages {37--51}. Kluwer Acad. Publ., Dordrecht, 2001.

\bibitem[DJ05]{DJ05}
M.~Dumnicki and W.~Jarnicki.
\newblock {New effective bounds on the dimension of a linear system in $\mathbb
  P^2$}.
\newblock arXiv:math/0505183, 2005.

\bibitem[Dum06]{D06}
M.~Dumnicki.
\newblock {Reduction method for linear systems of plane curves with base fat
  points }.
\newblock arXiv:math/0606716, 2006.

\bibitem[Eck05]{Eck05b}
Thomas Eckl.
\newblock Seshadri constants via lelong numbers.
\newblock preprint math.AG/0508561, to be published in Math.~Nachr., August
  2005.

\bibitem[Har03]{Harb03}
Brian Harbourne.
\newblock Seshadri constants and very ample divisors on algebraic surfaces.
\newblock {\em J. Reine Angew. Math.}, 559:115--122, 2003.

\bibitem[HR03]{HR03}
B.~Harbourne and J.~Ro{\'e}.
\newblock {Computing multi-point Seshadri constants on P2 }.
\newblock preprint, arXiv:math/0309064v3, 2003.

\bibitem[HR04]{HR04}
Brian Harbourne and Joaquim Ro{\'e}.
\newblock Linear systems with multiple base points in {$\Bbb P\sp 2$}.
\newblock {\em Adv. Geom.}, 4(1):41--59, 2004.

\bibitem[HR05]{HR05}
B.~Harbourne and J.~Ro{\'e}.
\newblock Multipoint {S}eshadri constants on {$\Bbb P\sp 2$}.
\newblock {\em Rend. Sem. Mat. Univ. Politec. Torino}, 63(1):99--102, 2005.

\bibitem[Mir99]{Mir99}
R.~Miranda.
\newblock {Linear Systems of plane curves}.
\newblock {\em Notices AMS}, 46:192--201, 1999.

\bibitem[Nag59]{Nag59}
M.~Nagata.
\newblock {On the 14-th problem of Hilbert}.
\newblock {\em Amer.J.Math}, 81:766--772, 1959.

\bibitem[Sch07]{Sch07}
F.~Sch\"uller.
\newblock Ein neuer ansatz zur harbourne-hirschowitz-vermutung.
\newblock diploma thesis, Universit\"at zu K\"oln, 2007.
\newblock http://www.mi.uni-koeln.de/\~{}kebekus/teaching/diplomarbeiten.html.

\bibitem[STG02]{ST02}
Tomasz Szemberg and Halszka Tutaj-Gasi{\'n}ska.
\newblock General blow-ups of the projective plane.
\newblock {\em Proc. Amer. Math. Soc.}, 130(9):2515--2524 (electronic), 2002.

\bibitem[TG03]{Tut03}
Halszka Tutaj-Gasi{\'n}ska.
\newblock A bound for {S}eshadri constants on {${\Bbb P}\sp 2$}.
\newblock {\em Math. Nachr.}, 257:108--116, 2003.

\bibitem[Xu94]{Xu95}
Geng Xu.
\newblock Curves in {${\bf P}\sp 2$} and symplectic packings.
\newblock {\em Math. Ann.}, 299(4):609--613, 1994.

\end{thebibliography}

\end{document}